\documentclass[english]{article}
\usepackage[T1]{fontenc}
\usepackage[latin9]{inputenc}
\usepackage{color}
\usepackage{amsthm}
\usepackage{amsmath}
\usepackage{amssymb}
\usepackage{esint}

\makeatletter
\numberwithin{equation}{section}
  \theoremstyle{definition}
  \newtheorem{defn}{\protect\definitionname}[section]
  \theoremstyle{plain}
  \newtheorem{thm}{\protect\theoremname}[section]
  \theoremstyle{remark}
  \newtheorem{rem}{\protect\remarkname}[section]
  \theoremstyle{plain}
  \newtheorem{prop}{\protect\propositionname}[section]
  \theoremstyle{plain}
  \newtheorem{lem}{\protect\lemmaname}[section]

\makeatother

\usepackage{babel}
  \providecommand{\definitionname}{Definition}
  \providecommand{\lemmaname}{Lemma}
  \providecommand{\propositionname}{Proposition}
  \providecommand{\remarkname}{Remark}
\providecommand{\theoremname}{Theorem}

\begin{document}

\title{The Signature of a Rough Path: Uniqueness }

\author{Horatio Boedihardjo%
\thanks{Department of Mathematics and Statistics, University of Reading. The
main part of his contribution on this paper was made while at Oxford-Man
Institute, University of Oxford. %
} , Xi Geng%
\thanks{Mathematical Institute and Oxford-Man Institute, University of Oxford %
}, Terry Lyons$^{\dagger}$ and Danyu Yang%
\thanks{Oxford-Man Institute, University of Oxford %
}}
\maketitle
\begin{abstract}
In the context of controlled differential equations, the signature
is the exponential function on paths. B. Hambly and T. Lyons proved
that the signature of a bounded variation path is trivial if and only
if the path is tree-like. We extend Hambly-Lyons' result and their
notion of tree-like paths to the setting of weakly geometric rough
paths in a Banach space. At the heart of our approach is a new definition
for reduced path and a lemma identifying the reduced path group with
the space of signatures. 
\end{abstract}

\section{Introduction }

In K.T. Chen's work \cite{Chen's first paper } on the cohomology
of the loop space, he defined and systematically studied the formal
series of iterated integrals 
\begin{equation}
S\left(x\right)=1+\sum_{i_{1}}\int_{0}^{T}\mathrm{d}x_{t_{1}}^{i_{1}}X_{i_{1}}+\sum_{i_{1},i_{2}}\int_{0}^{T}\int_{0}^{t_{2}}\mathrm{d}x_{t_{1}}^{i_{1}}\mathrm{d}x_{t_{2}}^{i_{2}}X_{i_{1}}\cdot X_{i_{2}}+\ldots\label{eq:signature as iterated integrals}
\end{equation}
where $x:\left[0,T\right]\rightarrow\mathbb{R}^{d}$ is a path with
bounded variation and $X_{1},\ldots,X_{d}$ are formal non-commutative
indeterminates. After proving a homomorphism property of the map $S$
(\cite{Chen's first paper }, see (\ref{eq:homomorphism}) below),
he gave an argument \cite{Chen uniqueness } that the map $S$ restricted
to appropriate classes of paths is, up to translation and reparametrisation,
injective. Hambly and Lyons \cite{tree like}, motivated by the application
of the map $S$ in rough path theory, posed the following problem:
\begin{quotation}
\textit{How to characterise the kernel of the map $S$? }
\end{quotation}
Hambly and Lyons \cite{tree like} proved that for a bounded variation
path $x$, $S\left(x\right)=1$ if and only if $x$ is\textit{ tree-like}.
They conjectured that the result extends to weakly geometric rough
paths, a fundamental class of control paths for which controlled differential
equations can be defined. Their result directly implies that the space
of bounded variation paths, quotiented by the space of tree-like paths,
forms a group with respect to the concatenation operation. They called
this quotient space the \textit{reduced path group}. 

In \cite{LQ12}, LeJan and Qian answered a special case of Hambly-Lyons
conjecture. They proved that, when restricted on the complement of
a Wiener measure zero set, the map $S$ (defined using Stratonovich
integration) is injective. There has been a number of other partial
results for particular cases of weakly geometric rough paths (\cite{GZ14,BNQ13,BG14,BG13}).
A key observation in the proof of Hambly-Lyons and further refined
by LeJan and Qian is that the iterated integrals of $1$-forms along
a path is a linear functional of the signature of the path. It turns
out that a subtle variant of this idea, by considering the $1$-form
along the iterated integrals of the path work to prove Hambly-Lyons'
conjecture. 

To formulate the extension of Hambly-Lyons result, we must find the
correct notion of tree-like weakly geometric rough paths. Hambly-Lyons'
definition of tree-like path is inappropriate in the setting of weakly
geometric rough paths. In fact, it is easy to prove that if $x$ is
an injective path with finite $p$-variation ($p>1$), but not finite
$1$-variation, then $S\left(x\star\overleftarrow{x}\right)=\mathbf{1}$
but $x\star\overleftarrow{x}$ won't be tree-like in the sense of
Hambly-Lyons. A tree-like path $x$ (in the sense of Hambly-Lyons)
has the property that there exists a continuous function $h:\left[0,1\right]\rightarrow\mathbb{R}$,
$h_{t}\geq0$ for all $t\in\left[0,1\right]$, $h_{1}=h_{0}$ and
\[
h_{s}=h_{t}=\inf_{s\leq u\leq t}h_{u}\implies x_{s}=x_{t}.
\]
We will take an equivalent formulation of this property (see \cite{Duquesne coding of real tree,Hambly Lyons note on trees})
as the definition of tree-like path, as follows.
\begin{defn}
\label{tree like path}Let $E$ be a topological space. A continuous
path $x:\left[0,T\right]\rightarrow E$ is \textit{tree-like} if there
exists a $\mathbb{R}$-tree $\tau$, a continuous map $\phi:\left[0,T\right]\rightarrow\tau$
and a map $\psi:\tau\rightarrow E$ such that $\phi\left(0\right)=\phi\left(T\right)$
and $x=\psi\circ\phi$. 
\end{defn}
This definition of tree-like path is equivalent to Hambly-Lyons' definition
when the path has bounded variation. The tree-like paths are also
interesting in the context of homotopy. In Section 5.7 in \cite{LVY11},
T. Levy showed that a bounded variation path is tree-like if and only
if it is contractible to the constant path within its own image. 

We will use our Definition \ref{tree like path} of tree-like path
to extend Hambly-Lyons' result to weakly geometric rough paths in
Banach spaces. The assumptions on the tensor product, which is only
needed in the infinite dimensional case, and the issue of defining
the map $S$ for weakly geometric rough paths will be discussed in
Section 2.1. 
\begin{thm}
\label{main result}Let $V$ be a Banach space such that the tensor
powers of $V$ is equipped with a family of tensor norms satisfying
(\ref{eq:admissible 1}), (\ref{eq:admissible 2}) and (\ref{eq:tensor condition 3}).
Let $x$ be a weakly geometric rough path in $V$. Then $S\left(x\right)=\mathbf{1}$
if and only if $x$ is tree-like in the sense of Definition \ref{tree like path}. \end{thm}
\begin{rem}
If $S\left(x\right)=\mathbf{1}$, then we even know what a corresponding
$\tau$,$\phi$ and $\psi$ in Definition \ref{tree like path} are.
Here $\tau$ is the set of signatures equipped with a special metric
(see Theorem \ref{tree metric}), $\phi$ is the path $t\rightarrow S\left(x\right)_{0,t}$
and $\psi$ is the projection map $\pi^{\left(\lfloor p\rfloor\right)}$
if $x\in WG\Omega_{p}\left(V\right)$ (see Section 2.1 for notations). 
\end{rem}

\begin{rem}
Our proof does not make use of Hambly-Lyons' result and hence in the
case $p=1$, we have given a new and simple proof of Hambly-Lyons'
theorem in \cite{tree like}.
\end{rem}
Theorem \ref{main result} has the same implications for weakly geometric
rough paths as Hambly-Lyons' for bounded variation paths. As part
of the proof (see Lemma \ref{existence and uniqueness of reduced path}),
we will also show that the map $S$ is an isomorphism from the \textit{reduced
paths} (an analytic notion) to signatures (an algebraic structure).
T. Lyons and W. Xu \cite{LX14} have proposed inversion schemes for
recovering the reduced paths from signatures for $C^{1}$ paths. 

The more difficult implication in Theorem \ref{main result} is ``$S\left(x\right)=\mathbf{1}$
implies $x$ is tree-like''. Our definition of tree-like weakly geometric
rough paths gives rise to a natural strategy of proof, namely to first
show that the space of signatures of paths has a $\mathbb{R}$-tree
structure. A key Lemma in our approach is to identify signatures with
injective paths on signatures (see Section 4.3). This Lemma translates
the natural $\mathbb{R}$-tree structure of the latter to a $\mathbb{R}$-tree
structure for the former (see Section 4.4 and Section 4.5). To identify
signatures with injective paths on signatures, we use that for sufficiently
smooth 1-form $\alpha$ and $N\in\mathbb{N}$, if paths $x$ and $y$
have the same signature and $S_{N}$ denotes the truncated signature
at degree $N$, then 
\[
\int\alpha\left(\mathrm{d}S_{N}\left(x\right)_{0,t}\right)=\int\alpha\left(\mathrm{d}S_{N}\left(y\right)_{0,t}\right)
\]
(see Section 4.1 and 4.2). That we need to integrate against the truncated
signature path $S_{N}\left(x\right)$ as opposed to merely integrating
against the path $x$ itself is an important new idea.

\section{The signature of a path}

We will follow Hambly-Lyons' \cite{tree like} and view the map $S$
as taking value in the formal series of tensors so that
\[
S\left(x\right)=\mathbf{1}+\int_{0}^{T}\mathrm{d}x_{t_{1}}+\int_{0}^{T}\int_{0}^{t_{2}}\mathrm{d}x_{t_{1}}\otimes\mathrm{d}x_{t_{2}}+\ldots.
\]
The formal series of tensors $S\left(x\right)$ is called the \textit{signature
}of the path\textit{ $x$.} The signature has a natural homomorphism
property with respect the operations of concatenation and reversal.
More precisely, given a Lie group $G$ and continuous paths $x:\left[0,T_{1}\right]\rightarrow G$
and $y:\left[0,T_{2}\right]\rightarrow G$, we define the concatenation
product $\star$ by 
\[
x\star y\left(t\right)=\begin{cases}
x\left(t\right), & t\in\left[0,T_{1}\right];\\
x\left(T_{1}\right)y\left(0\right)^{-1}y\left(t-T_{1}\right), & t\in\left(T_{1},T_{1}+T_{2}\right]
\end{cases}
\]
and the reversal operation $\overleftarrow{\cdot}$ by 
\[
\overleftarrow{x}\left(t\right)=x\left(T_{1}-t\right),t\in\left[0,T_{1}\right].
\]
 K.T. Chen proved two fundamental algebraic properties of the map
$S$, which in the language of tensors can be stated as :
\begin{enumerate}
\item (K.T. Chen, \cite{Chen's first paper })The map $S$ satisfies 
\begin{equation}
S\left(x\star y\right)=S\left(x\right)\otimes\left(y\right);\; S\left(x\right)\otimes S\left(\overleftarrow{x}\right)=\mathbf{1}.\label{eq:homomorphism}
\end{equation}
In the rough path literature, the first identity in (\ref{eq:homomorphism})
is now known as \textit{Chen's identity}.
\item (K.T. Chen, \cite{Chen Lie series })The natural logarithm of $S$,
in the space of non-commutative formal power series, is a Lie series. 
\end{enumerate}
Hambly-Lyons' characterisation of the kernel of the map $S$ implies
that the tree-like relation $\sim$ , defined for continuous bounded
variation paths $x$ and $y$ by
\[
x\sim y\iff x\star\overleftarrow{y}\mbox{ is tree-like},
\]
is an equivalence relation. Moreover, the space of continuous bounded
variation paths in $\mathbb{R}^{d}$, quotiented by the relation $\sim$,
is a group with respect to the binary operation $\star$ and the inverse
$\overleftarrow{\cdot}$. Our main result Theorem \ref{main result}
implies that the same holds with bounded variation paths replaced
by weakly geometric rough paths.

\subsection{Setting for rough path theory}

We briefly recall the notations and settings in rough path theory,
which will be identical to that in Lyons-Qian's book \cite{Lyons Qian}. 

Let $V$ be a Banach space. Let $\otimes$ be a tensor product such
that the tensor powers of $V$,$\left(V^{\otimes n}:n\geq1\right)$,
is equipped with a family $\left(\left\Vert \cdot\right\Vert _{V^{\otimes n}}:n\geq1\right)$
of norms satisfying: 

1. for $m,n\in\mathbb{N}$ and for all $u\in V^{\otimes m}$ and $v\in V^{\otimes n}$,
\begin{equation}
\left\Vert u\otimes v\right\Vert _{V^{\otimes\left(n+m\right)}}\leq\left\Vert u\right\Vert _{V^{\otimes m}}\left\Vert v\right\Vert _{V^{\otimes n}};\label{eq:admissible 1}
\end{equation}

2. for any permutation $\sigma$ of $\left\{ 1,\ldots,n\right\} $,
\begin{equation}
\left\Vert v_{1}\otimes\ldots\otimes v_{n}\right\Vert _{V^{\otimes n}}=\left\Vert v_{\sigma\left(1\right)}\otimes\ldots\otimes v_{\sigma\left(n\right)}\right\Vert _{V^{\otimes n}};\label{eq:admissible 2}
\end{equation}

3. for any bounded linear functionals $f$ on $V^{\otimes m}$ and
$g$ on $V^{\otimes n}$, there exists a unique bounded linear functional,
denoted as $f\otimes g$, on $V^{\otimes\left(m+n\right)}$ such that
for all $u\in V^{\otimes m}$ and $v\in V^{\otimes n}$, 
\begin{equation}
f\otimes g\left(u\otimes v\right)=f\left(u\right)g\left(v\right).\label{eq:tensor condition 3}
\end{equation}
A family of tensor norms satisfying conditions 1. and 2. are called
a family of admissible tensor norms (see Definition 1.25 in \cite{lyons st flours}).
By convention, we define $V^{\otimes0}$ by $\mathbb{R}$. The projective
tensor product $\left\Vert \cdot\right\Vert _{\vee}$, defined for
$v\in V^{\otimes n}$ by 
\[
\left\Vert v\right\Vert _{\vee}=\inf\left\{ \sum_{k}\left|v_{1}^{k}\right|\ldots\left|v_{n}^{k}\right|:v=\sum_{k}v_{1}^{k}\otimes\ldots\otimes v_{n}^{k},\; v_{i}^{k}\in V\right\} ,
\]
satisfies conditions 1.-3 (Section 5.6.1 in \cite{Lyons Qian}). We
will use the shorthand $\left\Vert \cdot\right\Vert $ to denote $\left\Vert \cdot\right\Vert _{V^{\otimes i}}$
and $ab$ to denote $a\otimes b$ where there is no confusion. 

Let $T\left(\left(V\right)\right)$ be the formal series of tensors
(Definition 2.4 \cite{lyons st flours}) and $T^{\left(n\right)}\left(V\right)$
be the truncated tensor algebra up to degree $n$ (Definition 2.5
\cite{lyons st flours}). Let $\pi_{n}$ and $\pi^{\left(n\right)}$
denote, respectively, the natural projection map from $T\left(\left(V\right)\right)$
onto $V^{\otimes n}$ and $T^{\left(n\right)}\left(V\right)$. Let
$\mathcal{L}\left(\left(V\right)\right)$ denote the Lie formal series
over $V$ (Definition 2.2 \cite{lyons st flours}). Let $\tilde{T}\left(\left(V\right)\right)$
and $\tilde{T}^{\left(n\right)}\left(V\right)$ denote the subspaces
of $T\left(\left(V\right)\right)$ and $T^{\left(n\right)}\left(V\right)$
respectively such that $\pi_{0}\left(a\right)=1$ for all $a\in\tilde{T}\left(\left(V\right)\right)$
or $\tilde{T}^{\left(n\right)}\left(V\right)$. Let $G^{\left(*\right)}=\exp\left(\mathcal{L}\left(\left(V\right)\right)\right)$
be the space of group-like elements (p37 \cite{lyons st flours})
and $G^{\left(n\right)}=\pi^{\left(n\right)}\left(G^{\left(*\right)}\right)$
denote the free nilpotent Lie group of step $n$ (p37 \cite{lyons st flours}).
We will equip $G^{\left(n\right)}$ with the metric 
\[
d\left(a,b\right)=\max_{i\in\left\{ 1,\ldots,n\right\} }\left\Vert \pi_{i}\left(a^{-1}b\right)\right\Vert ^{\frac{1}{i}}.
\]

\begin{defn}
\label{Girc}Let $G_{p.r.c}^{\left(*\right)}$ (p.r.c. for positive
radius of convergence) denote respectively the element $a$ in $G^{\left(*\right)}$
such that 
\[
\max_{i\in\mathbb{N}}\left\Vert \pi_{i}\left(a\right)\right\Vert ^{\frac{1}{i}}<\infty.
\]

\end{defn}
We will equip $G_{p.r.c.}^{\left(*\right)}$ with the metric 
\[
d\left(a,b\right)=\max_{i\in\mathbb{N}}\left\Vert \pi_{i}\left(a^{-1}b\right)\right\Vert ^{\frac{1}{i}}
\]
 (see Lemma 1.1 in \cite{infinite dimensional note} for a proof of
the symmetric property of $d$). Let $E$ be a metric space. We say
a continuous function $x:\left[0,T\right]\rightarrow E$ has finite
$p$-variation if 
\[
\left\Vert x\right\Vert _{p-var}=\sup_{\mathcal{P}}\left(\sum_{t_{i}\in\mathcal{P}}d\left(x_{t_{i}},x_{t_{i+1}}\right)^{p}\right)^{\frac{1}{p}}<\infty
\]
where the supremum is taken over all partitions $\mathcal{P}$ of
$\left[0,T\right]$. 
\begin{defn}
\label{rough path definition}The space of weakly geometric rough
paths, $WG\Omega_{p}\left(V\right)$, is the set of all continuous
functions from a compact interval $\left[0,T\right]$ to $G^{\left(\lfloor p\rfloor\right)}$
with finite $p$-variation.
\end{defn}
For $x,y\in WG\Omega_{p}\left(V\right)$ such that $x_{0}=y_{0}$,
we define the $p$-variation metric 
\[
d_{p-var}\left(x,y\right)=\max_{1\leq i\leq\lfloor p\rfloor}\sup_{\mathcal{P}}\left(\sum_{t_{j}\in\mathcal{P}}\left\Vert \pi_{i}\left(x_{t_{j}}^{-1}x_{t_{j+1}}-y_{t_{j}}^{-1}y_{t_{j+1}}\right)\right\Vert ^{\frac{p}{i}}\right)^{\frac{i}{p}}.
\]
We will use $\mathbf{1}$ to denote the identity element with respect
to $\otimes$ in $T\left(\left(V\right)\right)$. 
\begin{prop}
\label{extension theorem}(Extension Theorem, Theorem 2.2.1 \cite{MR1654527}
and Corollary 3.9 in \cite{Cass infinite dimensional rp}) Let $x\in WG\Omega_{p}\left(V\right)$.
There exists a unique continuous path $S\left(x\right)_{0,\cdot}:\left[0,T\right]\rightarrow G_{p.r.c}^{\left(*\right)}$
with finite $p$-variation such that $S\left(x\right)_{0,0}=\mathbf{1}$
and $\pi^{\left(\lfloor p\rfloor\right)}\left(S\left(x\right)_{0,t}\right)=x_{0}^{-1}x_{t}$.
We will call $S\left(x\right)_{0,T}$ the signature of $x$. 
\end{prop}
We will often omit the subscript and use the shorthand $S\left(x\right)$
for the signature of $x$. We will also use $S_{N}\left(x\right)$
to denote $\pi^{\left(N\right)}\left(S\left(x\right)\right)$.

\section{Tree-like paths have trivial signature}

\subsection{Preliminary definitions}
\begin{defn}
Let $V$ be a topological space and let $x:\left[a,b\right]\rightarrow V$,
$\tilde{x}:\left[c,d\right]\rightarrow V$ be continuous paths taking
value in $V$. We say $x$ is a\textit{ reparametrisation} of $\tilde{x}$
if there exists a homeomorphism $\sigma$ of $\left[c,d\right]$ onto
$\left[a,b\right]$ such that $x_{\sigma\left(t\right)}=\tilde{x}_{t}$
for all $t$. 
\end{defn}

\begin{defn}
Let $\tau$ be a $\mathbb{R}$-tree and $a,b\in\tau$, we will use
the notation $\left[a,b\right]$ to denote the unique (up to reparametrisation)
injective continuous function $x:\left[0,T\right]\rightarrow\tau$
on some compact interval $\left[0,T\right]$ such that $x_{0}=a$
and $x_{T}=b$.
\end{defn}

\begin{defn}
Let $V$ be a topological space. A \textit{rooted loop} in $V$ is
a continuous function $x:\left[s,t\right]\rightarrow V$ ($s\leq t$)
such that $x_{s}=x_{t}$. The element $x_{s}$ is known as the \textit{root}
of $x$.
\end{defn}

\begin{defn}
\label{partial order}Let $\tau$ be a $\mathbb{R}$-tree and $r\in\tau$.
We may define a partial order $\preceq$ with respect to $r$ on $\tau$
by 
\begin{equation}
a\preceq b\iff\left[r,a\right]\subseteq\left[r,b\right].\label{eq:partial order}
\end{equation}

\end{defn}

\subsection{The central case }
\begin{lem}
\label{piecewise geodesic tree-like}Let $\tau$ be a $\mathbb{R}$-tree
and $\phi:\left[0,T\right]\rightarrow\tau$ be a rooted loop. Suppose
there exists a partition $\mathcal{P}=\left(t_{0},\ldots,t_{n}\right)$
of $\left[0,T\right]$ such that if $t_{i}$, $t_{i+1}$ are adjacent
points in $\mathcal{P}$, then $\phi|_{\left[t_{i},t_{i+1}\right]}$
is (not necessarily strictly) monotone with respect to the root of
$\phi$. If $\psi:\tau\rightarrow G^{\lfloor p\rfloor}$ is such that
$\psi\circ\phi\in WG\Omega_{p}\left(V\right)$, then \textup{$\psi\circ\phi$}
has trivial signature. \end{lem}
\begin{proof}
We will prove by induction on $\left|\mathcal{P}\right|$. In the
case $\left|\mathcal{P}\right|=2$, as $\phi|_{\left[0,T\right]}$
is monotonic and $\phi\left(0\right)=\phi\left(T\right)$, $\phi$
is forced to be constant. In particular, $S\left(\psi\circ\phi\right)=\mathbf{1}$.
For the induction step, let $\tau_{\max}\in\phi\left(\mathcal{P}\right)$
be such that there does not exists $s\in\phi\left(\mathcal{P}\right)$,
$s\succ\tau_{\max}$. Let $t_{i}\in\mathcal{P}$ be such that $\phi\left(t_{i}\right)=\tau_{\max}$.
Since $\phi\left(t_{i-1}\right)\preceq\phi\left(t_{i}\right)$, $\phi\left(t_{i+1}\right)\preceq\phi\left(t_{i}\right)$
and the set $\left\{ t|t\preceq\phi\left(t_{i}\right)\right\} $ is
totally ordered, we may assume, without loss of generality, that $\phi\left(t_{i-1}\right)\preceq\phi\left(t_{i+1}\right)$.
Let $t^{\prime}\in\left[t_{i-1},t_{i}\right]$ be such that $\phi\left(t^{\prime}\right)=\phi\left(t_{i+1}\right)$.
Then $\phi|_{\left[0,t^{\prime}\right]\cup\left[t_{i+1},T\right]}$
is piecewise monotone with respect to the partition $\mathcal{P}\backslash\left\{ t_{i}\right\} $.
Therefore, by induction hypothesis, 
\begin{equation}
S\left(\psi\circ\phi|_{\left[0,t^{\prime}\right]\cup\left[t_{i+1},T\right]}\right)=\mathbf{1}.\label{eq:trivial signature without a point}
\end{equation}
As $\phi|_{\left[t^{\prime},t_{i}\right]}$ and $\phi|_{\left[t_{i},t_{i+1}\right]}$
are the unique injective curves connecting $\phi\left(t_{i}\right)$
and $\phi\left(t_{i+1}\right)$ with opposite orientation, by the
homomorphism property of signature (see Lemma\textbf{\textcolor{red}{{}
}}1.3 in \cite{infinite dimensional note}), 
\begin{eqnarray*}
 & S\left(\psi\circ\phi|_{\left[t^{\prime},t_{i}\right]}\right)\otimes S\left(\psi\circ\phi|_{\left[t_{i},t_{i+1}\right]}\right)=\mathbf{1} & .
\end{eqnarray*}
Hence $S\left(\psi\circ\phi|_{\left[t^{\prime},t_{i+1}\right]}\right)=\mathbf{1}$
which implies, by (\ref{eq:trivial signature without a point}) and
Chen's identity, that $S\left(\psi\circ\phi\right)=\mathbf{1}$. 
\end{proof}

\subsection{Reducing to the central case }

We will need the following result of general topology from \cite{erase loop},
which allows us to erase loops from a continuous path to obtain an
injective continuous path.
\begin{lem}
\label{Existence of simple}(R. Börger \cite{erase loop})Let $X$
be a Hausdorff space and let $\varphi:\left[0,1\right]\rightarrow X$
be continuous with $\varphi\left(0\right)\neq\varphi\left(1\right)$.
Then there exist a closed subset $A\subset\left[0,1\right]$, a continuous
and order-preserving map $q:\left[0,1\right]\rightarrow\left[0,1\right]$,
and an injective continuous map $\psi:\left[0,1\right]\rightarrow X$
with the following properties:

(i)$\psi\left(0\right)=\varphi\left(0\right)$,$\psi\left(1\right)=\varphi\left(1\right)$.

(ii)$\psi\circ q|_{A}=\varphi|_{A}$. 

(iii)$q|_{A}:A\rightarrow\left[0,1\right]$ is surjective. \end{lem}
\begin{rem}
The Lemma holds with $\left[0,1\right]$ replaced by any interval
$\left[0,T\right]$, $T>0$. 
\end{rem}
\begin{proof} [Proof of the "tree-like paths have trivial signature" part of Theorem 2.1]Using
the notation in Definition \ref{tree like path}, let the functions
$\phi:\left[0,T\right]\rightarrow\tau$ and $\psi:\tau\rightarrow G^{\left(\lfloor p\rfloor\right)}$
be a factorisation for the tree-like path $x$. Let $\preceq$ be
the natural partial order (see Definition \ref{partial order}) with
respect to the root of $\phi$. For any $a,b\in\tau$, we define $a\wedge b$
to be the unique element of $\tau$ such that 
\begin{eqnarray*}
\left[\phi\left(0\right),a\wedge b\right] & = & \left[\phi\left(0\right),a\right]\cap\left[\phi\left(0\right),b\right]
\end{eqnarray*}
(see Lemma 2.3 in \cite{Chi01} for the existence of $a\wedge b$).
Let $\mathcal{P}$ be a partition of $\left[0,T\right]$. Let 
\[
B=\left\{ \phi\left(t_{i_{1}}\right)\wedge\ldots\wedge\phi\left(t_{i_{n}}\right):t_{i_{1}},\ldots,t_{i_{n}}\in\mathcal{P},\; n\leq\left|\mathcal{P}\right|\right\} .
\]
The set $B$ can be interpreted as the set of branched points in the
subtree of $\tau$ spanned by $\phi\left(\mathcal{P}\right)$. Note
that for any $b_{1},b_{2}\in B$ we have $b_{1}\wedge b_{2}\in B$.
Define a sequence $\left(s_{i}\right)$ by $s_{0}=0$, 
\[
s_{i+1}=\inf\left\{ v>s_{i}:\phi\left(v\right)\in B\backslash\left\{ \phi\left(s_{i}\right)\right\} \right\} .
\]
By the continuity of $\phi$ and the finiteness of $B$, $\left(s_{i}\right)$
is a finite sequence. Let $\mathcal{P}^{\prime}=\left(s_{i}\right)$.
For each $s_{i}$, we construct an injective path $x_{\cdot}^{s_{i},s_{i+1}}$
in $\tau$ in the following way: if we apply Lemma \ref{Existence of simple}
to erase loops from the path 
\begin{equation}
\varphi\left(t\right)=\left(\phi|_{\left[s_{i},s_{i+1}\right]},\psi\circ\phi|_{\left[s_{i},s_{i+1}\right]}\right)\in\tau\times V,\label{eq:to erase}
\end{equation}
we obtain a continuous injective path $\eta$ in $\tau\times V$.
The path $x_{\cdot}^{s_{i},s_{i+1}}$ is the projection of $\eta$
onto $\tau$. Define $\phi^{\prime}:\left[0,T\right]\rightarrow\tau$
so that for each $s_{i}\in\mathcal{P}^{\prime}$ 
\begin{eqnarray*}
\phi^{\prime}\left(t\right) & = & \phi\left(t\right),\; t\in\mathcal{P}^{\prime}\\
 & = & x_{t}^{s_{i},s_{i+1}},\; s_{i}\leq t\leq s_{i+1}
\end{eqnarray*}
and let $x^{\mathcal{P}^{\prime}}:=\psi\circ\phi^{\prime}$. As the
self-intersection of $\varphi$ in (\ref{eq:to erase}) coincides
with the self-intersection with of $\phi|_{\left[s_{i},s_{i+1}\right]}$,
the path $x^{\mathcal{P}^{\prime}}|_{\left[s_{i},s_{i+1}\right]}$
is a natural projection of $\eta$ (see line below \ref{eq:to erase})
onto $V$. Therefore, $x^{\mathcal{P}^{\prime}}|_{\left[s_{i},s_{i+1}\right]}$
is continuous for all $i$. 

We now show that if $s_{i-1},s_{i}$ are adjacent points in $\mathcal{P}^{\prime}$,
then either $\phi\left(s_{i-1}\right)\preceq\phi\left(s_{i}\right)$
or $\phi\left(s_{i}\right)\preceq\phi\left(s_{i-1}\right)$. As stated
in Lemma 2.1 in \cite{Chi01}, the image of the continuous path $\phi|_{\left[s_{i},s_{i+1}\right]}$
in a $\mathbb{R}$-tree must contain $\left[\phi\left(s_{i}\right),\phi\left(s_{i+1}\right)\right]$
and in particular the element $\phi\left(s_{i}\right)\wedge\phi\left(s_{i-1}\right)$.
As $\phi\left(s_{i}\right)\wedge\phi\left(s_{i-1}\right)\in B$, by
the construction of the sequence $\left(s_{i}\right)$, $\phi\left(s_{i}\right)\wedge\phi\left(s_{i-1}\right)$
must be either equal to $\phi\left(s_{i}\right)$ or $\phi\left(s_{i+1}\right)$.

In particular, $\phi^{\prime}$ is piecewise monotone. Since the $p$-variation
of $x^{\mathcal{P}^{\prime}}$ is dominated by the $p$-variation
of $x$, $x^{\mathcal{P}^{\prime}}\in WG\Omega_{p}\left(V\right)$.
Therefore, $x^{\mathcal{P}^{\prime}}$ satisfies the assumptions of
the central case, Lemma \ref{piecewise geodesic tree-like}, and hence
has trivial signature. Let $\mathcal{P}_{n}$ be a sequence of partitions
such that $\left|\mathcal{P}_{n}\right|\rightarrow0$ as $n\rightarrow\infty$.
Let $\mathcal{P}_{n}^{\prime}$ be the corresponding sequence constructed
as above. Trivially, we have $\left\Vert x^{\mathcal{P}_{n}^{\prime}}\right\Vert _{p-var}\leq\left\Vert x\right\Vert _{p-var}$.
As $\left|\mathcal{P}_{n}^{\prime}\right|\rightarrow0$, $x^{\mathcal{P}_{n}^{\prime}}$
converges uniformly to $x$. By Lemma 1.5 in \cite{infinite dimensional note},
there exists a subsequence $x^{\mathcal{P}_{n_{k}}^{\prime}}$ such
that $S\left(x^{\mathcal{P}_{n_{k}}^{\prime}}\right)\rightarrow S\left(x\right)$
as $k\rightarrow\infty$. As $S\left(x^{\mathcal{P}_{n_{k}}^{\prime}}\right)=\mathbf{1}$
for all $k$, the result follows. \end{proof}

\section{Paths with trivial signature are tree-like}

\subsection{A special functional on signature}

The following is a key ingredient in the proof of our main result.
It establishes a relation between a weakly geometric rough path $x\in WG\Omega_{p}\left(V\right)$
and its signature other than the one given in the Extension Theorem
\ref{extension theorem}. 
\begin{lem}
\label{signature determines extended signatures as lemma}(Integration
of $1$-form is a functional of signature)Let $x,y\in WG\Omega_{p}\left(\mathbb{R}^{d}\right)$,
$\pi_{1}\left(x_{0}\right)=\pi_{1}\left(y_{0}\right)=0$ and $S\left(x\right)=S\left(y\right)$.
Then for any $N\in\mathbb{N}$ and any $C_{c}^{K}$ 1-form $\psi$
on $\mathbb{R}^{d}$ with $K>p-1$, we have 
\begin{equation}
\int\psi\left(\mathrm{d}x_{u}\right)=\int\psi\left(\mathrm{d}y_{u}\right).\label{signature determines extended signatures}
\end{equation}
\end{lem}
\begin{proof}
Suppose that 
\[
\psi\left(\mathrm{d}x\right)=\sum_{i=1}^{d}\psi_{i}\left(x\right)\mathrm{d}x^{i}.
\]
Assume for now that $\left\{ \psi_{i}\right\} $ were all polynomial
functions. Then each $\psi_{i}\left(\pi_{1}\left(x_{u}\right)\right)$
can be expressed as a linear functional (independent of path $x$
and time $u$) of $S\left(x\right)_{0,u}$ (see Theorem 2.15 \cite{lyons st flours}).
Since 
\[
S\left(x\right)=\mathbf{1}+\int_{0}^{T}S\left(x\right)_{0,t}\otimes\mathrm{d}x_{t},
\]
the integral $\int\psi\left(\pi_{1}\left(x_{u}\right)\right)\mathrm{d}x_{u}$
is also a linear functional, independent of $x$, of $S\left(x\right)$
and the desired result holds when $\psi_{i}$ are polynomials for
all $i$. It now suffices to note that functions in $C_{c}^{K}$ can
be approximated by polynomials in the Lip($K$)-norm (\cite{BBL})
and the map 
\[
\psi\rightarrow\int\psi(dz)
\]
is continuous in the Lip($K$)-norm for $K>p-1$ (\cite{FV10}, Theorem
10.50).
\end{proof}

\subsection{Finite dimensional projection of the signature path}

As pointed out in the introduction, the class of functionals 
\[
\left\{ x\rightarrow\int\psi\left(\mathrm{d}x\right):\psi\;\mbox{smooth 1-forms}\right\} 
\]
is insufficient to separate signatures. Indeed, if $x$ and $y$ are
any weakly geometric rough paths, then for any smooth 1-form, the
additivity of integral will imply that 
\[
\int\psi\left(\mathrm{d}\left(x\star y\right)\right)=\int\psi\left(\mathrm{d}\left(y\star x\right)\right),
\]
whereas in general $S\left(x\star y\right)\neq S\left(y\star x\right)$.
Therefore, we will consider a larger class of functionals, namely
\[
\left\{ x\rightarrow\int\psi\left(\mathrm{d}S\left(x\right)_{0,t}\right):\psi\;\mbox{smooth 1-forms}\right\} ,
\]
if we could at all make sense of the integral. However, the finite
dimensional nature of Lemma \ref{signature determines extended signatures as lemma}
forces us to use finite dimensional projection. We first truncate
the path in $G_{p.r.c.}^{\left(*\right)}$ to a path in $G^{\left(n\right)}\left(V\right)$
and show that a path in $G^{\left(n\right)}\left(V\right)$ with finite
$p$-variation can be lifted to a $p$-weakly geometric rough path.
Then we project the infinite dimensional space $V$ to $\mathbb{R}^{d}$
using a linear map $\Phi$. In Section 4.2.1, we will make sense of
the integral 
\[
\int\psi\left(\mathrm{d}\Phi\circ S_{N}\left(x\right)_{0,t}\right)
\]
for a bounded linear functional $\Phi$ on $\bigoplus_{i=1}^{N}V^{\otimes i}$,
a smooth 1-form $\psi$, a weakly geometric rough path $x$ and $N\in\mathbb{N}$.
In Section 4.2.2., we will show that for any two disjoint pieces of
signature paths, there will be a finite dimensional projection so
that their images remain separated.

\subsubsection{Integration against the signature path}

Let $x$ be a weakly geometric rough path. In order to apply Lemma
\ref{signature determines extended signatures as lemma} to $S_{N}\left(x\right)$,
we need to lift $S_{N}\left(x\right)_{0,\cdot}$ as a weakly geometric
rough path. Let $W=\bigoplus_{i=1}^{N}V^{\otimes i}$. We will implicitly
identify $W^{\otimes n}$ with $\bigoplus_{i_{1},\ldots,i_{n}=1}^{N}V^{\otimes\left(i_{1}+\ldots+i_{n}\right)}$.
Let $\pi_{i_{1},\ldots,i_{n}}$ denote the projection of $\bigoplus_{i_{1},\ldots,i_{n}=1}^{N}V^{\otimes\left(i_{1}+\ldots+i_{n}\right)}$
to the component $\left(i_{1},\ldots,i_{n}\right)$. We will equip
$W^{\otimes n}$ with the norm
\begin{equation}
\left\Vert v\right\Vert _{W^{\otimes n}}=\sum_{1\leq i_{1},\ldots,i_{n}\leq N}\left\Vert \pi_{i_{1},\ldots,i_{n}}\left(v\right)\right\Vert _{V^{\otimes\left(i_{1}+\ldots+i_{n}\right)}}.\label{eq:norm definition}
\end{equation}
 It is easy to see that the family of tensor norms $\|\cdot\|_{W^{\otimes n}}$
is admissible. In this section, we will use the notation $\pi_{k}$
to denote the projection from $T\left(\left(W\right)\right)$ onto
$W^{\otimes k}$. 
\begin{lem}
\label{canonical lift}Let $N\in\mathbb{N}$. There exists a map $\mathcal{J}:WG\Omega_{p}\left(V\right)\rightarrow WG\Omega_{p}\left(\bigoplus_{i=0}^{N}V^{\otimes i}\right)$
such that for all $x\in WG\Omega_{p}\left(V\right)$:

1. $\pi_{1}\left(\mathcal{J}\left(x\right)\right)=S_{N}\left(x\right)_{0,\cdot}$;

2. If $x,y\in WG\Omega_{p}\left(V\right)$ is such that $S\left(x\right)=S\left(y\right)$,
then $S\left(\mathcal{J}\left(x\right)\right)=S\left(\mathcal{J}\left(y\right)\right)$.\end{lem}
\begin{proof}
Suppose for now that $x$ is a path with bounded variation and $S\left(x\right)=\left(1,X^{1},X^{2},\ldots\right)$
where $X^{i}\in V^{\otimes i}$. Then we may define $\mathcal{J}$
by condition 1. in the Lemma. We define:

1. for $v_{1}\in V^{\otimes k_{1}},\ldots,v_{n}\in V^{\otimes k_{n}}$,
a map $F_{m_{1},\ldots,m_{n}}\left(v_{1},\ldots,v_{n}\right)$ on
$V^{\otimes\left(m_{1}+\ldots+m_{n}\right)}$ so that for all $w_{1}\in V^{\otimes m_{1}},\ldots,w_{n}\in V^{\otimes m_{n}}$,
\begin{eqnarray*}
 &  & F_{m_{1},\ldots,m_{n}}\left(v_{1},\ldots,v_{n}\right)\left[w_{1}\otimes\ldots\otimes w_{m_{1}+\ldots+m_{n}}\right]\\
 & = & v_{1}\otimes w_{1}\otimes\ldots\otimes w_{m_{1}}\otimes v_{2}\otimes w_{m_{1}+1}\otimes\ldots\otimes w_{m_{1}+m_{2}}\\
 &  & \otimes\ldots\otimes v_{n}\otimes w_{m_{1}+\ldots+m_{n-1}+1}\otimes\ldots\otimes w_{m_{1}+\ldots+m_{n}};
\end{eqnarray*}
2. for a permutation $\sigma$ on $\left\{ 1,\ldots,n\right\} $,
a map on $V^{\otimes n}$, also denoted by $\sigma$, by
\[
\sigma\left(v_{1}\otimes\ldots\otimes v_{n}\right)=v_{\sigma\left(1\right)}\otimes\ldots\otimes v_{\sigma\left(n\right)};
\]
3. for $j_{1},\ldots,j_{n}\in\mathbb{N}$, $OS\left(j_{1},\ldots,j_{n}\right)$
as the set of ordered shuffles (see p72 \cite{lyons st flours}). 

By Chen's identity (\ref{eq:homomorphism}), 
\begin{eqnarray}
 &  & \pi_{i_{1},\ldots,i_{n}}\left(S\left(\mathcal{J}\left(x\right)\right)_{s,t}\right)\label{eq:definition of lift}\\
 & = & \int_{s<s_{1}<\ldots<s_{n}<t}\mathrm{d}X_{0,s_{1}}^{i_{1}}\otimes\ldots\otimes\mathrm{d}X_{0,s_{n}}^{i_{n}}\\
 & = & \sum_{j_{n}=1}^{i_{n}}\ldots\sum_{j_{1}=1}^{i_{1}}F_{j_{1},\ldots,j_{n}}\left(X_{0,s}^{i_{1}-j_{1}},\ldots,X_{0,s}^{i_{n}-j_{n}}\right)\left[\int_{s<s_{1}<\ldots<s_{n}<t}\mathrm{d}X_{s,s_{1}}^{j_{1}}\otimes\ldots\otimes\mathrm{d}X_{s,s_{n}}^{j_{n}}\right]\nonumber \\
 & = & \sum_{j_{n}=1}^{i_{n}}\ldots\sum_{j_{1}=1}^{i_{1}}F_{j_{1},\ldots,j_{n}}\left(X_{0,s}^{i_{1}-j_{1}},\ldots,X_{0,s}^{i_{n}-j_{n}}\right)\left[\sum_{\pi\in OS\left(j_{1},\ldots,j_{n}\right)}\pi\left(X_{s,t}^{j_{1}+\ldots+j_{n}}\right)\right].\label{eq:signature}
\end{eqnarray}
The final identity (\ref{eq:signature}) uses identity (4.9) in \cite{lyons st flours}. 

We first restrict our attention to the finite dimensional case. For
$x\in WG\Omega_{p}\left(\mathbb{R}^{d}\right)$, we will define $\mathcal{J}\left(x\right)$
so that (\ref{eq:definition of lift}) holds. To show that $\mathcal{J}\left(x\right)$
is a weakly geometric rough path, we observe that if $x$ were to
have bounded variation, then $S\left(\mathcal{J}\left(x\right)\right)$
will satisfy Chen's identity and lie in $G_{p.r.c}^{\left(*\right)}$
($G_{p.r.c}^{\left(*\right)}$ here refers to group-like elements
over $\bigoplus_{i=0}^{N}V^{\otimes i}$ instead of $V$). For $x\in WG\Omega_{p}\left(\mathbb{R}^{d}\right)$,
let $p^{\prime}>p$ and let $x_{n}$ be a sequence of bounded variation
paths converging in $d_{p^{\prime}-var}$ to $x$ (see Corollary 8.26
in \cite{FV10} for the existence of such sequence). As $S\left(x_{n}\right)\rightarrow S\left(x\right)$
as $n\rightarrow\infty$ (Corollary 9.11 in \cite{FV10}), each $S\left(\mathcal{J}\left(x_{n}\right)\right)$
satisfies Chen's identity (\ref{eq:homomorphism}) and lies in $G_{p.r.c.}^{\left(*\right)}$.
We see by taking limit that $S\left(\mathcal{J}\left(x\right)\right)$
will still satisfy Chen's identity and lie in $G_{p.r.c}^{\left(*\right)}$.
Moreover, for all terms in the sum in (\ref{eq:signature}) $j_{1}+\ldots+j_{n}\geq n$,
which implies that $S\left(\mathcal{J}\left(x\right)\right)$ has
finite $p$-variation. Therefore, $\mathcal{J}$ maps $WG\Omega_{p}\left(V\right)$
to $WG\Omega_{p}\left(\bigoplus_{i=0}^{N}V^{\otimes i}\right)$. The
expression (\ref{eq:signature}) gives not just the expression for
$\mathcal{J}\left(x\right)$ but also for the signature of $\mathcal{J}\left(x\right)$,
which in particular implies the signature of the map $\mathcal{J}\left(x\right)$
is determined by the signature of $x$. The infinite dimensional case
of this result is included as Lemma 1.2 in \cite{infinite dimensional note} \end{proof}
\begin{lem}
\label{linear map on signatures}Let $W$ be a Banach space and $\Phi:W\rightarrow\mathbb{R}^{d}$
be a continuous linear functional on $W$. Then there exists a map
$\mathbf{F}:WG\Omega_{p}\left(W\right)\rightarrow WG\Omega_{p}\left(\mathbb{R}^{d}\right)$
such that for all $x\in WG\Omega_{p}\left(W\right)$:

1. $\pi_{1}\left(\mathbf{F}\left(x\right)\right)=\Phi\left(\pi_{1}\left(x\right)\right);$

2. If $x,y\in WG\Omega_{p}\left(W\right)$ is such that $S\left(x\right)=S\left(y\right)$,
then $S\left(\mathbf{F}\left(x\right)\right)=S\left(\mathbf{F}\left(y\right)\right)$. \end{lem}
\begin{proof}
For any linear map $\Phi:W\rightarrow\mathbb{R}^{d}$, by the admissiblity
conditions (\ref{eq:admissible 1}) and (\ref{eq:tensor condition 3})
of the tensor product, we may continuously extend $\mathbf{\Phi}$
to a bounded linear operator on $T^{\left(N\right)}\left(W\right)$
such that for $w_{1},\ldots,w_{N}\in W$, 
\[
\mathbf{\Phi}\left(w_{1}\otimes\ldots\otimes w_{N}\right)=\mathbf{\Phi}\left(w_{1}\right)\otimes\ldots\otimes\mathbf{\Phi}\left(w_{N}\right).
\]
Let $x\in WG\Omega_{p}\left(W\right)$. As $\Phi$ is a bounded linear
operator and the family of tensor norms on $\left(W^{\otimes n}:1\leq n\leq N\right)$
satisfies the admissibility conditions (\ref{eq:admissible 1}) and
(\ref{eq:tensor condition 3}), $\Phi\left(x\right)$ has finite $p$-variation.
As $\Phi$ is a homomorphism with respect to $\otimes$, for all $t\geq0$,
$\Phi\left(x_{t}\right)$ lies in the $\lfloor p\rfloor$-step free
nilpotent Lie group $G^{\lfloor p\rfloor}$ over $\mathbb{R}^{d}$.
Therefore, $\Phi\left(x\right)\in WG\Omega_{p}\left(\mathbb{R}^{d}\right)$.
By construction, $\pi_{1}\left(\Phi\left(x\right)\right)=\Phi\left(\pi_{1}\left(x\right)\right)$.
Moreover, again by the homomorphism property of $\Phi$ and admissibility
conditions (\ref{eq:admissible 1}) and (\ref{eq:tensor condition 3})
of the tensor norms, we have 
\[
\Phi\left(S_{N}\left(x\right)\right)=S_{N}\left(\Phi\left(x\right)\right)
\]
which implies property 2. in the Lemma. 
\end{proof}

\subsubsection{Separation of signature paths}

The following two lemmas together will tell us that with a careful
choice of truncation or finite dimensional projection, the images
of disjoint signature paths will remain disjoint. 
\begin{lem}
\label{finite dimensional simpleness} Let $\Delta=\left\{ \left(s,t\right)\in\left[0,T\right]^{2}:s\leq t\right\} $.
Let $S:\left[0,T\right]\rightarrow G_{p.r.c}^{\left(*\right)}$ be
an injective path. Then for any $\varepsilon>0,$ there exists $N(\varepsilon)\in\mathbb{N},$
such that $\pi^{\left(N\right)}\left(S_{s}\right)\neq\pi^{\left(N\right)}\left(S_{t}\right)$
for every $N\geq N(\varepsilon)$ and $(s,t)\in\Delta$ with $|t-s|\geq\varepsilon.$\end{lem}
\begin{proof}
Let $\Delta_{\varepsilon}=\{(s,t)\in\Delta:\ t-s\geq\varepsilon\}.$
For each $(s,t)\in\Delta_{\varepsilon},$ since $S\left(x\right)_{0,s}\neq S\left(x\right)_{0,t},$
there exists some $N_{s,t}\in\mathbb{N}$ such that 
\begin{equation}
S_{N_{s,t}}\left(x\right){}_{0,s}\neq S_{N_{s,t}}\left(x\right){}_{0,t}.\label{truncated simpleness}
\end{equation}
By continuity, (\ref{truncated simpleness}) holds in an open neighbourhood
of $(s,t).$ The result then follows easily from the compactness of
$\Delta_{\varepsilon}.$\end{proof}
\begin{lem}
\label{separation}Let $V$ be a Banach space, and $K,L$ be two disjoint
compact subsets of $V$. Then there exists $d\in\mathbb{N}$ and a
continuous linear functional $\Phi:\ V\rightarrow\mathbb{R}^{d}$,
such that $\Phi(K)$ and $\Phi(L)$ are disjoint in $\mathbb{R}^{d}.$\end{lem}
\begin{proof}
By Hahn-Banach theorem, for each $\kappa\in K$ and $\lambda\in L$,
there exists a bounded linear functional $f_{\kappa,\lambda}:V\rightarrow\mathbb{R}$
such that 
\[
f_{\kappa,\lambda}\left(\kappa\right)\neq f_{\kappa,\lambda}\left(\lambda\right).
\]
Fix $\lambda\in L$. By the continuity of $f_{\kappa,\lambda}$ for
each $\kappa\in K$ and the compactness of $K$, there exists $\kappa_{1},\ldots,\kappa_{j}\in K$
such that the continuous map 
\[
\Phi_{\lambda}\left(\cdot\right)=\left(f_{\kappa_{1},\lambda}\left(\cdot\right),\ldots,f_{\kappa_{1},\lambda}\left(\cdot\right)\right)
\]
sends the set $K$ and $\left\{ \lambda\right\} $ to disjoint sets.
By the continuity of $\Phi_{\lambda}$ for each $\lambda\in L$ and
the compactness of $L$, there exists $\lambda_{1},\ldots,\lambda_{j^{\prime}}\in L$
such that 
\[
\Phi\left(\cdot\right)=\left(f_{\kappa_{1},\lambda_{1}}\left(\cdot\right),\ldots,f_{\kappa_{j},\lambda_{1}}\left(\cdot\right),f_{\kappa_{1},\lambda_{2}}\left(\cdot\right),\ldots,f_{\kappa_{j},\lambda_{j^{\prime}}}\left(\cdot\right)\right)
\]
maps $K$ and $L$ to disjoint sets. 
\end{proof}

\subsection{Identifying signatures with signature paths}

The following Lemma lies at the heart of the proof of our main result,
Theorem \ref{main result}, and is interesting in its own right.
\begin{lem}
\label{existence and uniqueness of reduced path}(Existence and uniqueness
of reduced path)Let $S:\left[0,T\right]\rightarrow G_{p.r.c}^{\left(*\right)}$
be a continuous path with finite $p$-variation. There exists an injective
path $\tilde{S}:\left[0,\tilde{T}\right]\rightarrow G_{p.r.c.}^{\left(*\right)}$,
unique up to reparametrisation, such that $S_{T}=\tilde{S}_{\tilde{T}}$
and $S_{0}=\tilde{S}_{0}$. \end{lem}
\begin{rem}
In the case $p=1$, the weakly geometric rough path $\pi_{1}\left(\tilde{S}\right)$
is \textit{reduced }in the sense of Hambly-Lyons \cite{tree like},
meaning that it is the unique, up to translation and reparametrisation,
minimiser of the set 
\[
\left\{ \left\Vert x\right\Vert _{1-var}:x\in WG\Omega_{1}\left(V\right),S\left(x\right)_{0,1}=\tilde{S}_{0}^{-1}\tilde{S}_{1}\right\} .
\]
For weakly geometric rough paths, we define reduced path to be a weakly
geometric rough path $x$ such that the path $t\rightarrow S\left(x\right)_{0,t}$
is injective.\end{rem}
\begin{proof}
For the existence part, let $\tilde{S}:\left[0,T\right]\rightarrow G_{p.r.c}^{\left(*\right)}$
be the injective path obtained by applying Lemma \ref{Existence of simple}
to erase the loops in $S_{\cdot}$. Then by the order preserving property
of $q$ and (iii) in the loop erasing Lemma \ref{Existence of simple}
\begin{eqnarray}
\left\Vert \tilde{S}\right\Vert _{p-var} & \leq & \left\Vert S\right\Vert _{p-var},\label{eq:p-variation control}
\end{eqnarray}
which implies that $\tilde{S}$ has finite $p$-variation. 

For the uniqueness part, note the topological fact that two injective
continuous paths are reparametrisation of each other if and only if
they have the same starting point, ending point and image (see Lemma
26 in \cite{BNQ13}). Assume, for contradiction, that $t\rightarrow S_{t}$
and $t\rightarrow\tilde{S}_{t}$ are injective, $S_{0}=\tilde{S}_{0}$,
$S_{T}=\tilde{S}_{\tilde{T}}$ but $S$ and $\tilde{S}$ do not have
the same image. We may assume without loss of generality that $S_{0}=\tilde{S}_{0}=\mathbf{1}$.
Then there exists $s_{1}<s_{2}<t_{2}<t_{1}$ such that 
\begin{eqnarray}
S{}_{\left[s_{2},t_{2}\right]}\bigcap\left(S{}_{\left[0,s_{1}\right]\bigcup\left[t_{1},T\right]}\bigcup\mathrm{Im}\left(\tilde{S}{}_{\cdot}\right)\right) & = & \emptyset,\nonumber \\
S{}_{\left[s_{1},s_{2}\right]}\bigcap S{}_{\left[t_{2},t_{1}\right]} & = & \emptyset.\label{eq:disjoint}
\end{eqnarray}
It follows from the separation of finite dimensional projection results,
Lemma \ref{finite dimensional simpleness} and Lemma \ref{separation},
that there exists $N\in\mathbb{N}$ and a linear operator $\Phi:T^{\left(N\right)}\left(V\right)\rightarrow\mathbb{R}^{d}$
such that (\ref{eq:disjoint}) holds with $S$ and $\tilde{S}$ replaced
by $\Phi\big(\pi^{\left(N\right)}\left(S\right)\big)$ and $\Phi\big(\pi^{\left(N\right)}\left(\tilde{S}\right)\big)$. 

Let $U_{1},V_{1},U_{2},V_{2}$ be bounded open neighborhoods of $\Phi\big(\pi^{\left(N\right)}\left(S\right)\big){}_{\left[t_{2},t_{1}\right]}$,
$\Phi\big(\pi^{\left(N\right)}\left(S\right)\big){}_{\left[s_{1},s_{2}\right]}$,
$\Phi\big(\pi^{\left(N\right)}\left(S\right)\big){}_{\left[s_{2},t_{2}\right]}$,
$\Phi\big(\pi^{\left(N\right)}\left(S\right)\big){}_{\left[0,s_{1}\right]\bigcup\left[t_{1},T\right]}\bigcup\mathrm{Im}\left(\Phi\big(\pi^{\left(N\right)}\big(\tilde{S}_{\cdot}\big)\big)\right)$
respectively, such that 
\[
\overline{U_{1}}\bigcap\overline{V_{1}}=\emptyset,\ \overline{U_{2}}\bigcap\overline{V_{2}}=\emptyset.
\]
Let $f_{1},f_{2}\in C_{c}^{\infty}\left(\mathbb{R}^{d}\right)$ be
such that for $i=1,2$,
\[
f_{i}\left(\mathbf{X}\right)=\begin{cases}
1, & \mathbf{X}\in\overline{U_{i}};\\
0, & \mathbf{X}\in\overline{V_{i}}.
\end{cases}
\]
Consider the 1-form $\varphi=f_{2}\mathrm{d}f_{1}$. As the path $u\rightarrow\pi^{\left(\lfloor p\rfloor\right)}\left(S_{u}\right)$
is a weakly geometric rough path, Lemma \ref{canonical lift} and
Lemma \ref{linear map on signatures} together states that $\Phi\big(\pi^{\left(N\right)}\left(S_{\cdot}\right)\big)$
can be canonically lifted as a weakly geometric rough path. Moreover,
the signature of $\Phi\big(\pi^{\left(N\right)}\left(S_{\cdot}\right)\big)$
is a function of $S_{T}$. It follows that the integration of $1$-form
against $\Phi\big(\pi^{\left(N\right)}\left(S_{u}\right)\big)$ can
be defined and 
\begin{eqnarray*}
\int_{0}^{T}\varphi\left(\mathrm{d}\Phi\big(\pi^{\left(N\right)}\left(S_{u}\right)\big)\right) & = & \int_{s_{2}}^{t_{2}}\mathrm{d}f_{1}\left(\Phi\big(\pi^{\left(N\right)}\left(S_{u}\right)\big)\right)=1,
\end{eqnarray*}
while 
\[
\int_{0}^{\tilde{T}}\varphi\left(\mathrm{d}\Phi\big(\pi^{\left(N\right)}\left(\tilde{S}_{u}\right)\big)\right)=0.
\]
This leads to a contradiction to Lemma \ref{signature determines extended signatures as lemma}
which states that the integral of $1$-form is a functional of the
signature. 
\end{proof}

\subsection{Completing the proof}
\begin{defn}
Let $\mathcal{S}_{p}$ denote the set of injective continuous paths
in $G_{p.r.c}^{\left(*\right)}$ with finite $p$-variation starting
at $\mathbf{1}$. Define a relation $\preceq$ on $\mathcal{S}_{p}$
by 
\[
x\preceq y\iff\exists t\geq0,\; x\mbox{ is a reparametrisation of }y|_{\left[0,t\right]}.
\]
\end{defn}
\begin{lem}
\label{properties of partially ordered set}The space $\left(\mathcal{S}_{p},\preceq\right)$
is a partially ordered set such that:

1. $\mathcal{S}_{p}$ has a least element $\mathbf{1}_{\cdot}:\left[0,0\right]\rightarrow\mathbf{1}$. 

2. For all $S\in\mathcal{S}_{p}$, the set $\left\{ \hat{S}\in\mathcal{S}_{p}:\hat{S}\preceq S\right\} $
is totally ordered.

3. For all $S,\tilde{S}\in\mathcal{S}_{p}$, there exists an element
$S\wedge\tilde{S}\in\mathcal{S}_{p}$, unique up to reparametrisation,
such that 
\begin{equation}
\left\{ \hat{S}\in\mathcal{S}_{p}:\hat{S}\preceq S,\hat{S}\preceq\tilde{S}\right\} =\left\{ \hat{S}\in\mathcal{S}_{p}:\hat{S}\preceq S\wedge\tilde{S}\right\} .\label{eq:min definition}
\end{equation}

4. The function $\left\Vert \cdot\right\Vert _{p-var}^{p}:S\rightarrow\left\Vert S\right\Vert _{p-var}^{p}$
has the property that $\left\Vert \mathbf{1}\right\Vert _{p-var}^{p}=0$
and, for all $S$ the restriction of $\left\Vert \cdot\right\Vert _{p-var}^{p}$
on the set $\left\{ \hat{S}\in\mathcal{S}_{p}:\hat{S}\preceq S\right\} $
is strictly increasing. \end{lem}
\begin{proof}
The only non-trivial statement is statement 3. The uniqueness follows
trivially from (\ref{eq:min definition}). We now show the existence.
Let 
\[
t=\sup\left\{ \hat{t}\in\left[0,T\right]:S_{\hat{t}}\in\tilde{S}_{\left[0,\tilde{T}\right]}\right\} .
\]
We first show that the inclusion $\supseteq$ in (\ref{eq:min definition})
holds with $S\wedge\tilde{S}$ replaced by $S|_{\left[0,t\right]}$.
By the continuity of $S$ and that $\tilde{S}_{\left[0,\tilde{T}\right]}$
is closed, there exists $\tilde{t}$ such that $S_{t}=\tilde{S}_{\tilde{t}}$.
This implies $S|_{\left[0,t\right]}$ is a reparametrisation of $\tilde{S}|_{\left[0,\tilde{t}\right]}$
by the uniqueness of reduced path, Lemma \ref{existence and uniqueness of reduced path}.
In particular, $S|_{\left[0,t\right]}\preceq\tilde{S}$ and the desired
inclusion follows. 

Conversely, if $\hat{S}\preceq S$ and $\hat{S}\preceq\tilde{S}$,
then there exists $\hat{t}$ such that $\hat{S}$ is a reparametrisation
of $S|_{\left[0,\hat{t}\right]}$ and hence $S_{\hat{t}}\in\tilde{S}_{\left[0,\tilde{T}\right]}$.
In particular, $\hat{t}\leq t$ and $\hat{S}\preceq S|_{\left[0,t\right]}$.\end{proof}
\begin{prop}
\label{tree metric}For $S^{\left(1\right)},S^{\left(2\right)}\in\mathcal{S}_{p}$,
define 
\begin{eqnarray*}
d\left(S^{\left(1\right)},S^{\left(2\right)}\right) & = & \left\Vert S^{\left(1\right)}\right\Vert _{p-var}^{p}+\left\Vert S^{\left(1\right)}\right\Vert _{p-var}^{p}-2\left\Vert S^{\left(1\right)}\wedge S^{\left(2\right)}\right\Vert _{p-var}^{p}.
\end{eqnarray*}
Then $\left(\mathcal{S}_{p},d\right)$ is a $\mathbb{R}$-tree. \end{prop}
\begin{proof}
By Proposition 3.10 in \cite{Tree partial order } (or see Lemma 1.7
in \cite{infinite dimensional note}), a partially ordered set satisfying
1.-4. in Lemma \ref{properties of partially ordered set} is a $\mathbb{R}$-tree
with metric $d$. 
\end{proof}
Let $\left(\mathcal{G}_{p},d_{\mathcal{G}_{p}}\right)$ denote the
metric space defined as the pushforward, under the map $\mathbf{P}:S\rightarrow S_{T}$
(sending a path to its value at terminal time), of the metric space
$\left(\mathcal{S}_{p},d\right)$. As a set, Lemma \ref{existence and uniqueness of reduced path}
implies that 
\[
\mathcal{G}_{p}=\left\{ S\left(x\right)_{0,T}:x\in WG\Omega_{p}\left(V\right)\right\} .
\]
The metric space $\left(\mathcal{G}_{p},d_{\mathcal{G}_{p}}\right)$
as the isometric image of a $\mathbb{R}$-tree is itself a $\mathbb{R}$-tree.
The injective map $\mathbf{P}$ induces naturally a partial order,
also denoted as $\preceq$, and an operation $\wedge$ on $\mathcal{G}_{p}$
satisfying (\ref{eq:min definition}) with $\mathcal{S}_{p}$ replaced
by $\mathcal{G}_{p}$.
\begin{lem}
\label{p-variation estimate}(Continuity estimate for the right concatenation)Let
$S:\left[0,T\right]\rightarrow G_{p.r.c.}^{\left(*\right)}$ has finite
$p$-variation. Then 
\begin{equation}
\left\Vert S\right\Vert _{p-var}^{p}-\left\Vert S|_{\left[0,s\right]}\right\Vert _{p-var}^{p}\leq\left(1+p\right)\left\Vert S\right\Vert _{p-var}^{p-1}\left\Vert S|_{\left[s,T\right]}\right\Vert _{p-var}.\label{eq:continuity of p-var}
\end{equation}
\end{lem}
\begin{proof}
Let $\mathcal{P}=\left(t_{0}<t_{1}<\ldots<t_{n}\right)$ be a partition
of $\left[0,T\right]$. Let $j$ be the last time in $\mathcal{P}$
such that $t_{j}\leq s$. Then 
\begin{eqnarray}
 &  & \sum_{i=0}^{n-1}d\left(S_{t_{i}},S_{t_{i+1}}\right)^{p}\nonumber \\
 & \leq & \sum_{i=0}^{j-1}d\left(S_{t_{i}},S_{t_{i+1}}\right)^{p}+d\left(S_{t_{j}},S_{s}\right)^{p}+d\left(S_{s},S_{t_{j+1}}\right)^{p}+\sum_{i=j+1}^{n-1}d\left(S_{t_{i}},S_{t_{i+1}}\right)^{p}\nonumber \\
 &  & +\left[d\left(S_{t_{j}},S_{t_{j+1}}\right)^{p}-d\left(S_{t_{j}},S_{s}\right)^{p}-d\left(S_{s},S_{t_{j+1}}\right)^{p}\right].\label{eq:intermediate}
\end{eqnarray}
By the mean value theorem and triangle inequality, 
\[
d\left(S_{t_{j}},S_{t_{j+1}}\right)^{p}-d\left(S_{t_{j}},S_{s}\right)^{p}\leq p\left\Vert S\right\Vert _{p-var}^{p-1}d\left(S_{s},S_{t_{j+1}}\right),
\]
which, together with (\ref{eq:intermediate}), implies (\ref{eq:continuity of p-var}). \end{proof}
\begin{lem}
\label{continuity of signature path}If $S:\left[0,T\right]\rightarrow G_{p.r.c}^{\left(*\right)}$
is continuous and has finite $p$-variation, then $S_{\cdot}$ is
continuous in $\left(\mathcal{G}_{p},d_{\mathcal{G}_{p}}\right)$. \end{lem}
\begin{proof}
We first argue that for all $s<t$, there is a $u\in\left[s,t\right]$
such that $S_{u}=S_{s}\wedge S_{t}$. By applying Lemma \ref{Existence of simple}
to erase loops from $S|_{\left[s,t\right]}$, we obtain an injective
path $\tilde{\mathbf{S}}$ with finite $p$-variation connecting $S_{s}$
and $S_{t}$. By the definition of $S_{s}\wedge S_{t}$, there is
an injective path $\hat{\mathbf{S}}$ with finite $p$-variation connecting
$S_{s}$ to $S_{s}\wedge S_{t}$ and then to $S_{t}$. By the uniqueness
of reduced path (Lemma \ref{existence and uniqueness of reduced path}),
$\hat{\mathbf{S}}$ must coincide with $\tilde{\mathbf{S}}$, implying
our claim. We now note that

\begin{eqnarray*}
d_{\mathcal{G}_{p}}\left(S_{s},S_{t}\right) & \leq & \left\Vert \mathbf{P}^{-1}\left(S_{u}\right)\star\overleftarrow{S|_{\left[s,u\right]}}\right\Vert _{p-var}^{p}-\left\Vert \mathbf{P}^{-1}\left(S_{u}\right)\right\Vert _{p-var}^{p}\\
 &  & +\left\Vert \mathbf{P}^{-1}\left(S_{u}\right)\star S_{\left[u,t\right]}\right\Vert _{p-var}^{p}-\left\Vert \mathbf{P}^{-1}\left(S_{u}\right)\right\Vert _{p-var}^{p}
\end{eqnarray*}
which converges to $0$ as $t\rightarrow s$ by the continuity estimate
for the right concatenation, Lemma \ref{p-variation estimate}. 
\end{proof}
\begin{proof} [Proof of the "Paths with trivial signature are tree-like" part of Theorem 2.1]Let
$x\in WG\Omega_{p}\left(V\right)$ be such that $S\left(x\right)_{0,T}=1$.
By the Extension Theorem (Proposition \ref{extension theorem}), $S\left(x\right)_{0,t}\in G_{p.r.c.}^{\left(*\right)}$
for all $t$. In the Definition \ref{tree like path} of tree-like
path, take $\tau$ to be $\left(\mathcal{G}_{p},d_{\mathcal{G}_{p}}\right)$
(which is a $\mathbb{R}$-tree as shown earlier in Proposition \ref{tree metric}),
$\phi\left(t\right)=S\left(x\right)_{0,t}$ and $\psi\left(z\right)=\pi_{\left\lfloor p\right\rfloor }\left(z\right)$.
Then by the definition of signature, $x_{t}=\pi_{\left\lfloor p\right\rfloor }\left(S\left(x\right)_{0,t}\right)$.
The continuity of $\phi$ has been shown in Lemma \ref{continuity of signature path}.\end{proof}

\section*{Acknowledgement}
All four authors gratefully acknowledge the support of ERC (Grant Agreement No.291244 Esig). The third author has also been supported by EPSRC (EP/F029578/1). We would like to thank Prof. Thierry Levy for giving detailed and useful suggestions for the first draft of this paper.

\end{document}